\documentclass[12pt]{amsart}

\usepackage{graphicx}
\usepackage{amsthm}
\usepackage{amsmath, amssymb}
\swapnumbers

\theoremstyle{plain}
\newtheorem{Thm}{Theorem}

\newtheorem{Coro}[Thm]{Corollary}
\newtheorem{Lem}[Thm]{Lemma}

\theoremstyle{definition}
\newtheorem{Def}[Thm]{Definition}

\newcommand{\mS}{\mathcal{S}}
\newcommand{\mC}{\mathcal{C}}

\begin{document}

\title{Heegaard splittings and open books}

\author{Jesse Johnson}
\address{\hskip-\parindent
        Department of Mathematics \\
        Oklahoma State University \\
        Stillwater, OK 74078 \\
        USA}
\email{jjohnson@math.okstate.edu}

\subjclass{Primary 57N10}
\keywords{Heegaard splitting, open book decomposition}

\thanks{This project was supported by NSF Grant DMS-1006369}

\begin{abstract}
We show that if the monodromy of an open book decomposition has sufficiently high displacement distance, acting on the loop and arc complex for a page, then it is the unique minimal Euler characteristic open book for the manifold.  In particular, we show that such an open book induces the unique (up to isotopy) minimal genus Heegaard surface for the manifold, and that this Heegaard surface has cyclic mapping class group.
\end{abstract}

\maketitle

An \textit{open book decomposition} of a compact, connected, closed, orientable 3-manifold $M$ is a pair $(L, \pi)$ where $L \subset M$ is a link and $\pi : (M \setminus L) \rightarrow S^1$ is a surface bundle map such that the closure of each page $\pi^{-1}(t)$ is an embedded surface with boundary $L$. There is a surface $F$ such that the closure of each page $\pi^{-1}(t)$ is homeomorphic to $F$ and the Euler characteristic of $(L, \pi)$ is defined as the Euler characteristic of $F$.  The link complement $M \setminus L$ can be identified with the interior of the quotient $F \times [0,1]$ by a homeomorphism $\phi : F \times \{0\} \rightarrow F \times \{1\}$, that restricts to the identity on $\partial F$. This map is called the \textit{monodromy} of $\pi$.

Choose a collection of points $m \subset \partial F$, with one point in each component of $\partial F$.  The \textit{marked curve complex} $\mC(F) = \mC(F, m)$ is a simplicial complex with each vertex representing a properly embedded, essential simple arc or loop disjoint from $m$, modulo isotopy disjoint from $m$. (In particular, the endpoints of an arc cannot cross a point in $m$.)  Two vertices bound an edge if they have disjoint representatives. If a set of vertices is pairwise connected by edges then the simplex bounded by these vertices is also in $\mC(F)$. The distance $d(u,v)$ between vertices of $\mC(F)$ is the number of edges in the shortest edge path from $u$ to $v$.

Because $\phi$ fixes $\partial F$ pointwise, $m \times [0,1]$ is sent to a collection of points in $L$ with one point in each component. Moreover, $\phi$ determines an isometry of $\mC(F)$, which we will also denote by $\phi :\mC(F) \rightarrow \mC(F)$.  The \textit{displacement} of $\phi$ is $d(\phi) = \min \{d(v, \phi(v))\}$, where the minimum is over all the vertices of $\mC(F)$.  Note that performing $\frac{1}{n}$ Dehn surgery on $L$ corresponds to Dehn twisting $\phi$ along a loop parallel to the boundary of $F$. If we choose a very high surgery coefficient, the displacement in the arc and curve complex will be determined by the displacement in the loop-only curve complex.

The displacement $d(\phi)$ is analogous to the Hempel distance for a Heegaard splitting.  Scharlemann and Tomova~\cite{st:dist} proved that a Heegaard splitting of a 3-manifold $M$ whose distance is greater than twice its genus is, up to isotopy, the unique minimal genus Heegaard splitting for $M$.  Bachman-Schleimer~\cite{bs:versus} introduced the displacement map to prove a similar Theorem relating surface bundles to Heegaard splittings. We will prove an analogous theorem for open book decompositions. (When reading the Theorem, recall that Euler characteristic is, in general, negative.)

\begin{Thm}
\label{mainthm}
Let $(L, \pi)$ be an open book decomposition of a 3-manifold $M$ with Euler characteristic $\chi$ whose monodromy has displacement $d$ on the curve complex for a page in $\pi$.  If $d > 8 - 4\chi$ then every open book decomposition of $M$ has Euler characteristic at most $\chi$ and $(L, \pi)$ is (up to isotopy) the unique open book with Euler characteristic $\chi$.
\end{Thm}

By Giroux's correspondence~\cite{giroux}, every open book decomposition for $M$ induces a contact structure on $M$ and every contact structure is induced by a family of open book decompositions. A contact structure will be Stein fillable if and only if the monodromy map of some open book inducing the contact structure can be written as a composition of right handed Dehn twists~\cite{akoz}\cite{giroux}\cite{loipi}. By choosing right handed Dehn twists along a collection of loops that cut the surface into disks, we can construct a pseudo-Anosov map on any surface with Euler characteristic at most $-2$. By taking a suitably high power we can find a map whose displacement is arbitrarily high. Along with Theorem~\ref{mainthm}, this implies that Stein fillable contact structures induced by minimal open books are quite common:

\begin{Coro}
\label{maincoro1}
For every surface $F$ with Euler characteristic $\chi(F) \leq -2$, there is a 3-manifold $M$ with a Stein fillable contact structure induced by a minimal open book decomposition whose pages are homeomorphic to $F$.
\end{Coro}

Note that Theorem~\ref{mainthm} bounds the Euler characteristic of alternate open book decompositions, but not their genus. For example, every 3-manifold has an open book decomposition with planar pages. If $M$ has a high genus, high distance open book decomposition then Theorem~\ref{mainthm} implies that every planar open book for $M$ will have a large number of link components.

The proof of Theorem~\ref{mainthm} uses the Heegaard splitting induced by the open book decomposition. A \textit{Heegaard splitting} is a triple $(\Sigma, H^-, H^+)$ where $\Sigma \subset M$ is a compact, closed, separating surface and $H^-, H^+ \subset M$ are handlebodies with disjoint interiors whose union is $M$ and whose common boundary is $\Sigma$.  To prove Theorem~\ref{mainthm}, we consider a Heegaard surface $\Sigma$ that consists of the union of $L$ and two pages of the surface bundle.  Because $F \times I$ is homeomorphic to a handlebody for $I$ an interval and $F$ a surface with boundary, the union of any two pages bounds two handlebodies in $M$ and thus determines a Heegaard splitting.  

The \textit{mapping class group} $Mod(M, \Sigma)$ of a Heegaard splitting is the group of automorphisms of $M$ that take $\Sigma$ onto itself, modulo isotopies of $M$ through homeomorphisms that keep $\Sigma$ on itself. The \textit{isotopy subgroup} consists of elements of $Mod(M, \Sigma)$ that are isotopy trivial on $M$. For an open book $(L, \pi)$, continuous rotation of the circle induce automorphisms of the surface bundle $M \setminus L$.  If $\Sigma$ coincides with the pre-images of two antipodal points in $S^1$ then a rotation by angle $n\pi$ will take $\Sigma$ to itself for every integer $n$ and interchange the two handlebodies for odd $n$.  We will call this a \textit{book rotation} of $\Sigma$.  These elements of the isotopy subgroup of $Mod(M, \Sigma)$ form a cyclic subgroup, which will be infinite order if and only if the monodromy map $\phi$ is infinite order (modulo isotopies of $F$ fixing $m$).  Theorem~\ref{mainthm} is a corollary of the following:

\begin{Thm}
\label{mainthm2}
Let $(L, \pi)$ be an open book decomposition of a 3-manifold $M$ with Euler characteristic $\chi$ whose monodromy has displacement $d$. If $d > 12 - 4\chi$ then the isotopy subgroup of the mapping class group of the Heegaard splitting induced by $(L, \pi)$ will consist of book rotations.
If $(\Sigma, H^-, H^+)$ is any genus $g$ Heegaard splitting with $d > 4g + 4$ then $\Sigma$ is a stabilization of the Heegaard splitting induced by $(L, \pi)$.
\end{Thm}

Note that a Heegaard splitting induced by an open book decomposition always has distance at most two, so Scharlemann-Tomova's theorem cannot be applied in this situation.  In the Heegaard surface, the high distance occurs on the two subsurfaces defined by the two pages of the open book. This makes the result more along the lines of the subsurface projection distance bound in~\cite{jmm}. However, the subsurface distance on each of the two subsurfaces is zero. It is only when taking into account the fact that the two subsurfaces are linked, as defined in~\cite{ms:holes}, that one can formulate a notion of high distance in the Heegaard surface.

The proof of Theorem~\ref{mainthm2} is based on a combination of thin position techniques and double sweep-out techniques.  These model the arguments in~\cite{st:dist},~\cite{bs:versus} and~\cite{me:stabs}, but with the sweep-out arguments modified and disguised as thin position arguments.  We use Hayashi-Shimokawa thin position~\cite{hs:thin} within the framework of axiomatic thin position~\cite{axiomatic}.

We review thin position and the machinery of axiomatic thin position and the complex of surface in Section~\ref{thinsect}. Flat surfaces, as introduced in~\cite{me:stabs2}, are defined and examined in Sections~\ref{flatsect},~\ref{esssect} and~\ref{diskbandsect}. We generalize the machinery for finding essential flat surfaces from~\cite{me:stabs2} to the present setting in Sections~\ref{indxonesect} and~\ref{indxtwosect}, then use these surfaces to find the distance bound in Sections~\ref{distsect} and~\ref{mosltlyhsect}. In the final Section, we bring these tools together to prove Theorems~\ref{mainthm} and~\ref{mainthm2}.

I thank Ken Baker for pointing out the connection between this work and contact structures.

\section{Thin position}
\label{thinsect}

We first review the definition of thin position for a surface with respect to a link $L \subset M$, in terms of the complex of surfaces.

A two-sided (possibly disconnected) surface $S \subset M$ is \textit{strongly separating} if the components of $M \setminus S$ can be labeled $+$ and $-$ so that each component of $S$ is the frontier of both a positive component and a negative component. A choice of such labels will be called a \textit{transverse orientation}.

We will consider isotopy classes of surfaces transverse to $L$, i.e. given two surfaces transverse to $L$, we will consider them isotopic if there is an isotopy in which each intermediate surface is transverse to $L$.  Moreover, two surfaces will be called \textit{sphere blind isotopic} if they are related by a sequence of isotopies (transverse to $L$) and the following moves: We will allow ourselves to add or remove sphere components that bound balls in $M$ disjoint from $L$ and the rest of the surface.  We also allow ourselves to attach a tube from any sphere component disjoint from $L$ to any other component, or to pinch off a sphere component (i.e. the inverse of adding a tube).

A \textit{compressing disk} for a surface $S \subset M$ transverse to $L$ is an embedded disk $D \subset M$ whose interior is disjoint from $S$ and $L$ and whose boundary is an essential loop in $S \setminus L$. A \textit{bridge disk} is a disk $D \subset M$ with interior disjoint from $S$ and $L$ whose boundary consists of an arc in $L$ and an (essential) arc with interior in $S \setminus L$. A \textit{K-disk} is either a compressing disk or a bridge disk (The letter `K' stands for knot.) and such a disk defines either a compression or bridge compression, respectively.

Following~\cite{axiomatic}, we define the \textit{complex of surfaces} $\mS(M, L)$ as the cell complex whose vertices are sphere blind isotopy classes of transversely oriented, strongly separating surfaces transverse to $L$. Edges connect each surface to the surfaces that result from compressing or bridge compressing it.  For each pair of disjoint compressions, there is a loop of up to four edges that correspond to compressing along the two disks in either order. We include a 2-cell in $\mS(M, L)$ bounded by this loop for each such pair and higher dimensional cells correspond to larger collections of pairwise disjoint compressions.  

Think of the vertices of $\mS(M, L)$ arranged by negative Euler characteristic plus number of components (ignoring sphere components). Compressions and boundary compressions reduce the negative Euler characteristic, so we can think of the corresponding edges as pointing down from the vertex. The \textit{descending link} $L_v$ of a vertex $v$ is the simplicialization of the subcomplex of its link spanned by the edges below the vertex. That is, we identify any two cells in this subcomplex with the same boundary. 

The result is a flag simplicial complex whose vertices correspond to compressing disk and bridge compressing disks for the surface, with edges connecting disjoint disks. In other contexts, this complex is called the \textit{disk complex} for the surface $S$ representing $v$. The \textit{index} of $v$ is defined as zero when the descending link $L_v$ is empty, and otherwise is equal to $i + 1$ where $\pi_i(L_v)$ is the first non-trivial homotopy group of the descending link. This definition of index was introduced by Dave Bachman~\cite{bach:index} (using the disk complex for the surface, but without the complex of surfaces) based on ideas first suggested by Hyam Rubinstein.

A bridge compression reduces the number of points of intersection between a surface and $L$, so if we isotope a surface $S$ (by an isotopy that is not necessarily transverse to $L$) to minimize its intersection number with $L$, then there will be no bridge compressions for the resulting surface.  Moreover, if $S$ is incompressible then there will be no compressions in the complement of $L$, so the vertex $v$ representing the resulting surface will have index zero (i.e. its descending link will be empty).  Thus we have the following Lemma:

\begin{Lem}
If $S$ is an incompressible surface for $M$ then $S$ is isotopic to an index-zero surface with respect to $L$.
\end{Lem}

Any Heegaard surface for $M$ can be isotoped to be disjoint from $L$, and moreover it can be isotoped so that $L$ is contained on either side, and $S$ is completely compressible on the side not containing $L$.  These isotopies and compressions determine an oriented path in $\mS(M, L)$ in which the Heegaard surface appears as a local maximum.  As in~\cite{axiomatic}, this path can be weakly reduced to a path with index-one maxima and index-zero minima. (This is one of the basic results of thin position, first shown by Gabai~\cite{gabai}, but we phrase it in terms of axiomatic thin position here for consistency.)

Such a path may not contain an interior local minimum, but it will always contain a local maximum, and this will correspond to a surface with genus at most that of $\Sigma$.  Thus we have the following:

\begin{Lem}
If $\Sigma$ is a genus $g$ Heegaard surface for $M$ then there is an index-one surface $S \subset M$ with respect to $L$ with genus at most $g$ that intersects $L$ non-trivially.  
\end{Lem}

If we can change one path in $\mS(M, L)$ to another by replacing edges in one path by new edges such that new and old edges bound a face then we say that the two paths are related by a \textit{face slide}. Any isotopy from a Heegaard surface $\Sigma$ to itself can be extended to an ambient isotopy of $M$ and determines a sequence of paths in $\mS(M, L)$, related by face slides.  By~\cite{axiomatic}, this sequence of paths can also be thinned, to a sequence consisting of a number of steps in which we replace two or more consecutive maxima with a single index-two maximum, then replace the index-two maximum with a new sequence of index-one maxima. 

Each path in this sequence has at least one maximum of genus $g$ and by following these maxima, we can find a path with index-two maxima and index-one and -zero maxima that corresponds to the isotopy. This is summed up as follows:

\begin{Lem}
\label{indextwopathlem}
If $\Sigma$ is a strongly irreducible genus $g$ Heegaard surface that is index-one with respect to $L$ then every isotopy of $\Sigma$ is defined by a (possibly constant) path in $\mS(M, L)$ with index-two maxima and index-zero and -one minima, all all isotopic in $M$ to $\Sigma$.
\end{Lem}

Each of these Lemmas implies the existence of a surface with non-trivial index with respect to $L$.  To prove Theorems~\ref{mainthm} and~\ref{mainthm2}, we will show that the existence of such a surface implies a bound on the displacement of the monodromy map of $\pi$.

\section{Flat surfaces}
\label{flatsect}

Flat surfaces were introduced in~\cite{me:stabs2} in order to get more control over double sweep-outs of a 3-manifold. Here, we adapt the idea to comparing sweep-outs and open book decompositions.

Given an open book $(L, \pi)$ as above, let $N \subset M$ be the complement of an open regular neighborhood of $L$.  Then $N$ is a compact 3-manifold whose boundary consists of one or more tori.  The restriction of $\pi$ to $N$ is a surface bundle map, which we will also denote by $\pi : N \rightarrow S^1$.  

A \textit{page} of $\pi$ is a level surface $F_t = \pi^{-1}(t)$ for $t \in S^1$.  A \textit{vertical annulus} $A \subset N$ is an annulus whose interior is transverse to the pages of $\pi$ and whose boundary consists of two loops in pages of $\pi$.  A \textit{vertical band} $B \subset N$ is a disk that intersects the pages of $\pi$ in parallel arcs, including two arcs in the boundary of $B$. The remaining two arcs of $\partial B$ are contained in $\partial N$ and transverse to the pages of $\pi$.  A \textit{horizontal subsurface} in $N$ is a subsurface of a page $F_t$.

The intersection of a surface $\Sigma \subset M$ with $N$ is a properly embedded surface in $S \subset N$.  We will say that a compact, orientable, properly embedded surface $S \subset N$ is \textit{flat} if $\partial S$ is essential in $\partial N$ and $S$ is the union of a collection of vertical annuli, vertical bands and horizontal subsurfaces such that any two vertical annuli or bands in $S$ are disjoint.  In particular, the boundaries of the vertical annuli must be boundary loops in horizontal subsurfaces.  The boundaries of the vertical bands must be contained in the union of $\partial N$ and the horizontal subsurfaces.  

\begin{Lem}
Every piecewise linear surface properly embedded in $N$ is isotopic to a flat surface.
\end{Lem}

\begin{proof}
If $S$ is a piecewise linear surface then (after isotoping $S$ slightly if necessary), the level sets of $\pi|_S$ will consist of simple closed curves, simple arcs and graphs in $S$.  We can isotope $S$ so as to make a regular neighborhood of each graph horizontal.  The complement of these horizontal subsurfaces will be foliated by simple closed curves and properly embedded arcs, and thus consist of pairwise disjoint vertical annuli and vertical bands.
\end{proof}

Orienting $S^1$ induces a collection of preferred normal vectors on each page $F_t$.  If $S$ is strongly separating and transversely oriented in $N$ then every horizontal subsurface $F \times \{a\}$ has a positive component of $N \setminus S$ on one side and a negative component on the other.  We will say that this subsurface \textit{faces up} if the preferred normal vectors point towards the positive component and \textit{faces down} if they point towards the negative component.  (Every subsurface will face either up or down.)

\begin{Def}
We will say that a flat surface $S$ is \textit{tight} if the following conditions hold:
\begin{itemize}
\item No horizontal annulus disjoint from $\partial N$ or disk with two arcs in $\partial N$ has one adjacent vertical annulus/band above it and the other below. 
\item For every horizontal disk subsurface with two or fewer arcs in the boundary, the closest adjacent horizontal subsurface faces the opposite way from the disk.
\item For every horizontal disk subsurface with three arcs in the boundary, the closest adjacent horizontal subsurface on the side with at least two vertical bands faces the opposite way from the disk.
\item For every horizontal annulus component disjoint from $\partial N$ with both adjacent vertical annuli above it (or below it), the closer adjacent horizontal subsurface faces the opposite way from the annulus.
\end{itemize}
\end{Def}

Note that this is a generalization of the definition of tight surfaces in a sweep-out in~\cite{me:stabs2}, with additional conditions because $F$ has non-empty boundary.

\begin{Lem}
\label{tightsurfacelem}
Every flat surface $S$ is isotopic to a tight surface $S'$ such that every horizontal subsurface of $S$ is sent into a horizontal subsurface of $S'$ or into a vertical band or annulus.
\end{Lem}

\begin{proof}
Let $S$ be a flat surface and let $E \subset S$ be a horizontal subsurface. In each of the cases described above, we will define an isotopy that eliminates the bad subsurface.

If $E$ is a horizontal annulus whose adjacent vertical annuli go in opposite directions, we can shrink $E$ to a single loop, turning the two vertical annuli into a single vertical annulus. If $E$ is a disk with two arcs in $\partial N$ and the adjacent bands going in opposite directions then we can employ a similar construction, shrinking the horizontal disk to an arc and combining the two vertical bands.

Otherwise, assume $E$ is a disk with zero or one arcs in $\partial N$, a disk with two arcs $\partial N$ and adjacent bands going the same direction or an annulus with both adjacent vertical annuli going in the same direction. If the horizontal subsurface $E'$ adjacent to $E$ faces the same way as $E$ then the projections of the two horizontal subsurfaces into $F$ must be disjoint (up to isotopy) and we can isotope $E$ into the same level as $E'$. The same is true if $E$ is a disk with three arcs in $\partial D$ and the closest adjacent subsurface on the same side as at least two of the vertical bands faces the same way as $E$.
\end{proof}

\section{Essential surfaces}
\label{esssect}

A flat surface $S$ will be called \textit{essential} if for every horizontal subsurface $E = F_t \cap S$, the complement $\partial E \setminus \partial F_t$ is a collection of essential loops and essential, properly embedded arcs in $F_t$.

\begin{Lem}
\label{indexzerolem}
Assume $S$ is the intersection of $N$ with an index-zero (with respect to $L$) surface in $M$ such that $S \setminus N$ is a collection of essential disks in $M \setminus N$.  If $S$ is tight then $S$ is essential.  
\end{Lem}

\begin{proof}
If $S$ is not essential then by definition there is a horizontal subsurface $F \subset \Sigma_t$ whose boundary contains a trivial loop.  Lemma~19 in~\cite{me:stabs2} states that if $S$ is tight and there is a horizontal loop in $S$ that is trivial in some $F_t$ then there is a horizontal loop that is essential in $S$ and trivial in some $F_t$. The Lemma generalizes directly to the case when $F_t$ has boundary, implying that if there is a horizontal arc or loop in $S$ that is trivial in $F_t$ then there is an essential loop in $S$ that is trivial in $F_t$. An innermost (in $\Sigma$) such loop bounds a disk in $N$ that defines a compressing disk for $S$, contradicting the assumption that $S$ is the intersection of $N$ with an index-zero surface.
\end{proof}

We will show in later sections that index-one and -two surfaces are isotopic to surfaces that are either essential or have a slightly more complicated form. A subset $S \cap F_t$ of a surface $S \subset N$ is a \textit{subsurface with a flipped square} if it consists of the union of a subsurface $E$ of $S$ and a disk $E'$ in $S$ such that $E \cap E'$ consists of four points.  Moreover, the intersection of the vertical annuli or bands just above $F_t$ with those just below $F_t$ is these same four points.  In such a subsurface, the disk $E'$ faces the direction opposite from $E$. We will say that $S \cap F_t$ has $n$ flipped squares if $S \cap F_t$ is the union of a subsurface of $S$ and $n$ disks, each of which intersects the subsurface in four points where vertical annuli and bands above and below that level intersect.
\begin{figure}[htb]
  \begin{center}
  \includegraphics[width=3.5in]{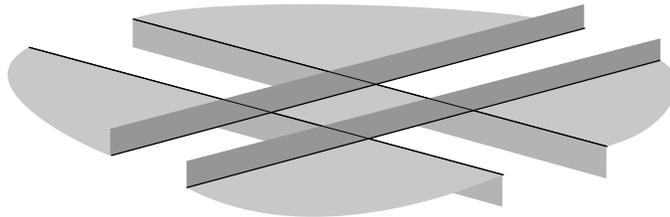}
  \caption{A flipped square is a horizontal disk that intersects the rest of the horizontal subsurface in four points and faces the opposite way.}
  \label{flippedfig}
  \end{center}
\end{figure}

Recall that a K-disk is either a compressing disk or bridge disk. If two K-disks for $S$ have disjoint boundaries in the complement of $L$ then either the disks define a face in the complex of surfaces or they consist of two bridge disks whose boundaries intersect in exactly two points in $L$.  If $S$ has a pair of horizontal disks $D_1$, $D_2$ induced by bridge disks that intersect in two points in $L$ then a regular neighborhood in $S$ of $\partial D_1 \cup \partial D_2$ and their two common boundary components is an annulus $A$, punctured twice by $L$ which we will call a \textit{one-bridge annulus}.  

\begin{Def}
A surface $S \subset N$ is an \textit{index-$n$ essential surface} if it consists of horizontal subsurfaces with $k$ flipped squares, $\ell$ index-one annuli, where $n = k + \ell$, and the horizontal boundaries of the remaining vertical arcs and loops are all essential.
\end{Def}

We will see below that higher index surfaces can be made essential only if we allow flipped squares and one-bridge annuli.  This is analogous to Bachman's result~\cite{bach:index} that topologically minimal surfaces can be made to intersect a family of incompressible surfaces essentially, as long as we allow some number of tangencies with the incompressible surfaces.

\section{Disks and band moves}
\label{diskbandsect}

To make an index-one or index-two tight surface $S \subset N$ essential, we will use its K-disks to define a sequence of moves of the following type:

Let $E_1, E_2 \subset S$ be horizontal subsurfaces of a tight surface $S \subset N$ in consecutive levels of $S^1$ such that (without loss of generality) $E_1$ is below $E_2$ and first assume the subsurfaces face opposite ways.  Let $I \subset S^1$ be the interval between the levels containing $E_1$ and $E_2$.  We can think of these as subsurfaces of $F$.  Let $\alpha$ be a properly embedded arc in $E_1$ such that the endpoints of $\alpha$ are in vertical bands or annuli that are above $E_1$. 

Because we can locally project $\alpha$ between pages, we will abuse notation and talk about the intersection of $\alpha$ with subsurfaces in other pages. Assume $\alpha \cap E_2$ is empty or consists of a regular neighborhood in $\alpha$ of one or both of its endpoints. Then the vertical band $\alpha \times I$ between the levels containing $E_1$ and $E_2$ forms a disk whose boundary consists of an arc in $S$ and a horizontal arc $\alpha'$ which is the closure of $\alpha \setminus E_2$.

The endpoints of $\alpha'$ are contained in one or two vertical annuli or bands above $E_2$.  Assume $I' \subset S^1$ is the interval above $I$ such that the level defined by the upper endpoint of $I'$ contains horizontal subsurfaces $E_3$ of $S$ facing the same way as $E_1$. Note that this implies $\alpha'$ will be disjoint from $E_3$. Then $\alpha' \times I'$ is a disk whose boundary intersects $S$ in two vertical arcs.  The union $D = (\alpha \times I) \cup (\alpha' \times I')$ is a disk whose boundary consists of an arc in $S$ and the horizontal arc parallel to $\alpha'$.  Let $S'$ be the result of isotoping $S$ across the disk $D$ as in Figure~\ref{bandmovefig}, then pulling the resulting flat surface tight.
\begin{figure}[htb]
  \begin{center}
  \includegraphics[width=5in]{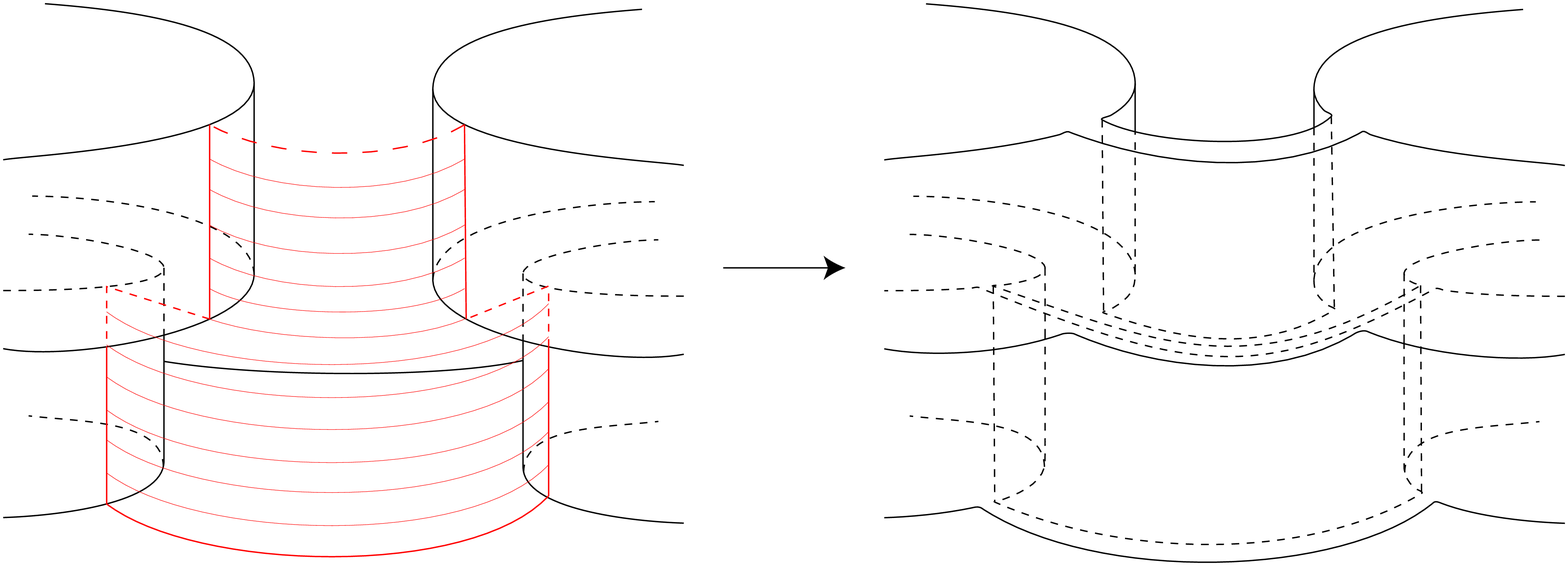}
  \put(-250,0){$\alpha$}
  \caption{The red disk defines a band move from the surface on the left to the surface on the right.}
  \label{bandmovefig}
  \end{center}
\end{figure}

\begin{Def}
We will say that $S'$ is the result of a \textit{band move} on $S$.
\end{Def}

A similar move can be defined under a number of weaker circumstances. Assume there is an essential subsurface $E' \subset E_2$ such that in $F$, the intersection $\alpha \cap E'$ is either empty or consists of interval neighborhoods of one or both endpoints of $\alpha$.  Then we can form a new flat surface as follows:  Isotope the subsurface $E'$ of $E_2$ down to the level defined by the midpoint of $I$. The arc $\alpha$ now intersects the horizontal subsurface above $E_1$ in a way that allows us to perform a band move, creating a non-empty up-facing level in the interior of $I$.  

If $E_1$, $E_2$ face the same way, we can imagine an empty subsurface between them. In this case the band move takes a subsurface from $E_1$ directly into $E_2$. Similarly, if $E_1$ and $E_2$ face opposite directions, but the next level $E_3$ faces the same way as $E_2$, we can imagine an empty subsurface just above $E_2$ and move a band from $E_1$ into this empty subsurface. In all these cases, we will also say that the resulting surface $S'$ is a \textit{band move} of $S$.

Before we begin using band moves in Section~\ref{indxonesect}, we need two technical lemmas.

\begin{Lem}
Assume $S$ is the intersection of $N$ with a surface $\Sigma \subset M$. If $S$ is a tight surface in $M$ then every bridge disk for $\Sigma$ can be isotoped so that its intersection with $\partial N$ is a level arc in some $\partial F_t$.
\end{Lem}

\begin{proof}
The complement $M \setminus (N \cup L)$ is homeomorphic to a torus cross an interval and the intersection of $\Sigma$ with this submanifold is a collection of properly embedded, vertical annuli. The intersection of the boundary of a bridge disk $D$ with $N$ consists of a collection of arcs with a total of two endpoints in $L$. The arcs that do not have endpoints in $L$ can be isotoped within $\Sigma$ out of $N$. We can further isotope $D$ to eliminate any disks of intersection disjoint from $\Sigma$.  After this isotopy, the intersection of $D$ with $M \setminus (N \cup L)$ will be a vertical band containing $\alpha$, so $D \cap \partial M$ will be a single arc parallel to $\alpha$ and can be isotoped to a level arc in $\partial N$.
\end{proof}

For any K-disk $D$ for $S$, we can isotope $D$ so that the restriction of $\pi$ to the interior of $D$ is a Morse function whose level sets consist of loops, properly embedded arcs and saddles.  (A saddle is a graph with one valence four vertex in the interior of $D$ and zero or more valence one vertices in the boundary.) 

A \textit{tetrapod} is a saddle with four vertices in $\partial D$.  Such a level $\tau$ cuts $D$ into four disks. If $D$ is a bridge disk then one of the four disks contains the arc $\partial D \cap N$. If three of the disks are disjoint from $\partial N$ and do not contain any tetrapods then we will say that $\tau$ is an \textit{outermost tetrapod}. As Schultens showed in~\cite{Schl}, if $D$ is not vertical or horizontal then there must be an outermost tetrapod. The Lemma below also appears in~\cite{me:stabs2} but we include it here for completeness.

\begin{Lem}
\label{tetrapodisotopylem}
If a disk $D$ for a tight surface $S$ containing an outermost tetrapod $\tau$ then there is a sequence of band moves of $S$ after which we can isotope $D$ in $N$ to eliminate $\tau$.
\end{Lem}

\begin{proof}
Let $D_1$, $D_2$, $D_3$ be the disks in the complement of $\tau$ that are disjoint from $\partial N$ and from all the tetrapods other than $\tau$, labelled so that $D_2$ is adjacent to both $D_1$ and $D_3$.  The boundary of $D$ consists of vertical arcs and horizontal arcs. If the restriction of $\pi$ to $D_1$ has a local maximum then we can push this maximum down to cancel it with the nearest saddle in $D_1$. This isotopy occurs within a ball in $N$, so if a portion of $S$ is in the way, we can isotope it down as well. Moreover, this isotopy can be extended to any other $K$ for $S$ without producing new tetrapods. If we repeat the process to eliminate all maxima and minima in $D_1$, $D_2$ and $D_3$, each disk $D_i$ will be foliated by arcs, with a single maximal or minimal arc $\alpha_i \subset D_i \cap \partial D$.

Without loss of generality, assume that the horizontal arc $\alpha_2$ is above $D_2$.  Then by construction, the horizontal arcs $\alpha_1$, $\alpha_3$ in $D_1 \cap \partial D$ and $D_3 \cap \partial D$, respectively, are below $D_1$ and $D_3$.  The tetrapod $\tau$ is contained in a page $F_a$ sitting between two horizontal subsurfaces in $S$.  Each disk $D_i$ intersects every level between $\alpha_i$ and $\tau$ in an arc parallel to the projection of $\alpha_i$.  If one of these levels contains a horizontal subsurface of $S$ then the projection of $\alpha_i$ will be disjoint from it or intersect it in a neighborhood of one or both of its endpoints.  Thus there is a band move of $S$ that moves each $\alpha_i$ past each level between it and $\tau$.  

In the surface $S'$ that results from these band moves for the three disks, the arcs $\alpha_1$, $\alpha_3$ sit in the level surface of $S'$ below $\tau$ and $\alpha_2$ sits in the level subsurface surface just above $\tau$.  The projections of $\alpha_1$, $\alpha_2$, $\alpha_3$ into the level surface $\Sigma \times \{a\}$ containing the tetrapod $\tau$ are parallel to subgraphs of $\tau$, so their projections are isotopic to pairwise disjoint arcs.  The vertical disk defined by $\alpha_2$ will thus intersect a regular neighborhood of $\alpha_1 \cup \alpha_3$ in the complement of a regular neighborhood of its endpoints. Thus we can perform a band move in which we split the subsurface containing $\alpha_1 \cup \alpha_3$ into two subsurfaces and pull $\alpha_2$ down past the top one containing $\alpha_1 \cup \alpha_3$. After this band move, the arc $\alpha_2$ is below $\alpha_1$, $\alpha_3$, so we can isotope $D$ to remove the tetrapod $\tau$.
\end{proof}

\section{Index-one surfaces}
\label{indxonesect}

If a tight surface $S$ contains a horizontal subsurface with a boundary loop or arc that is trivial in $F$ then an innermost (in $F$) such loop or arc defines a compressing disk or bridge disk (respectively) for $S$.  We will call these \textit{horizontal disks}.  Any two horizontal disks are disjoint away from $L$ and the bridge disks defined any two horizontal arcs intersect in one or two points in $L$. To make a tight surface $S$ into an essential surface, we must eliminate all horizontal disks via band moves.

\begin{Lem}
\label{indexonelem}
Every index-one surface is represented by an essential surface in $N$, possibly with one flipped square or one-bridge annulus.
\end{Lem}

\begin{proof}
Let $S$ be the intersection of $N$ with a surface representing an index-one vertex $v$ in $\mS(M, H)$. Let $D^-$, $D^+$ be a pair of K-disks in different components of the disk complex for $S$.  (In general, $D^-$ will be on the negative side of $S$ and $D^+$ on the positive side, but for the sake of generality that will be useful later, we will not specify this.)

If $D^+$ is not vertical or horizontal then it contains an outermost tetrapod and there is a sequence of band moves defining a sequence of tight surfaces $S_0,S_{1},\dots,S_{k'}$ that eliminate the tetrapod.  If the image of $D^+$ after these band moves is still not vertical or horizontal then we can extend this sequence by further band moves until $D^+$ contains no tetrapods, and is thus vertical or horizontal.  If the final image of $D^+$ is vertical then either $S$ is the boundary of a neighborhood of a horizontal loop or there is a final band move that makes $D^+$ horizontal. In the first case, $S$ cannot be index-one, so we will assume that we can make $D^+$ horizontal.

Let $S_0,\dots,S_k$ be the resulting sequence of tight surfaces. Define a similar sequence $S_0,S_{-1},\dots,S_{-\ell}$ based on the disk $D^-$.  If some $S_i$ is essential then we have found an essential representative of $S$. Otherwise, as noted above, each $S_i$ contains an essential loop or arc bounding a horizontal K-disk $D_i \subset F_t$.

The band move from $S_0$ to $S_1$ consists of three parts:  First, we are allowed to separate a horizontal subsurface of $S_0$ into two subsurfaces, along a collection of essential curves, producing $S'_0$.  Next, we move a band from a second horizontal subsurface past one of these, to get $S'_1$.  Finally, we make $S'_1$ tight to produce $S_1$.  Splitting a horizontal subsurface of $S_0$ does not create or eliminate any trivial vertical bands or annuli, so any horizontal disk for $S_0$ is isotopic to a horizontal disk for $S'_0$ and vice versa.  Similarly, pulling $S'_1$ tight does not produce new horizontal loops or arcs, so any horizontal disk for $S_1$ is isotopic to a horizontal disk for $S'_1$.

If there is a trivial loop or arc in $S'_0$ that is also contained in $S'_1$ then there is a horizontal K-disk $D'_0$ in both surfaces (and thus in both $S_0$ and $S_1$). Thus $D'_0$ is disjoint in $S$ from both $D_0$ and $D_1$ or coincides with one of these disks. 

If $D_0$ and $D'_0$ define bridge disks that intersect in two points in $L$ then they define a one-bridge annulus for $S$. If some other level loop or arc in $S'_0$ bounds a K-disk disjoint from both $D_0$ and $D'_0$ then we will replace $D'_0$ with this disk. Otherwise, the rest of the surface is essential so $S'_0$ is an index-one essential surface with one one-bridge annulus. A similar argument applies to the disks $D'_0$ and $D_1$.  If there are no one-bridge annuli then there are edges in the disk complex for $S$ from $D_0$ to $D'_0$ and from $D'_0$ to $D_1$.

We can repeat this argument for each sequential pair $D_i$, $D_{i+1}$. If we find a disjoint disk $D'_i$ for each, then we will have constructed a path in the disk complex for $S$ from $D^- = D_{-\ell}$ to $D^+ = D_k$, contradicting the assumption that these disks are in distinct path components. Thus there must be a value $i$ such that no horizontal disk is common to both $S_i$ and $S_{i+1}$. 

Let $E$ be the disk defining the band move from $S'_0$ to $S'_1$.  The move affects three horizontal levels, which we will label $F_- < F_0 < F_+$.  If we perform only the first half of the band move, which brings the band into the level of $F_0$, then the resulting surface $S''$ contains a flipped square in this level, as in Figure~\ref{flippedbandfig}.  
\begin{figure}[htb]
  \begin{center}
  \includegraphics[width=5in]{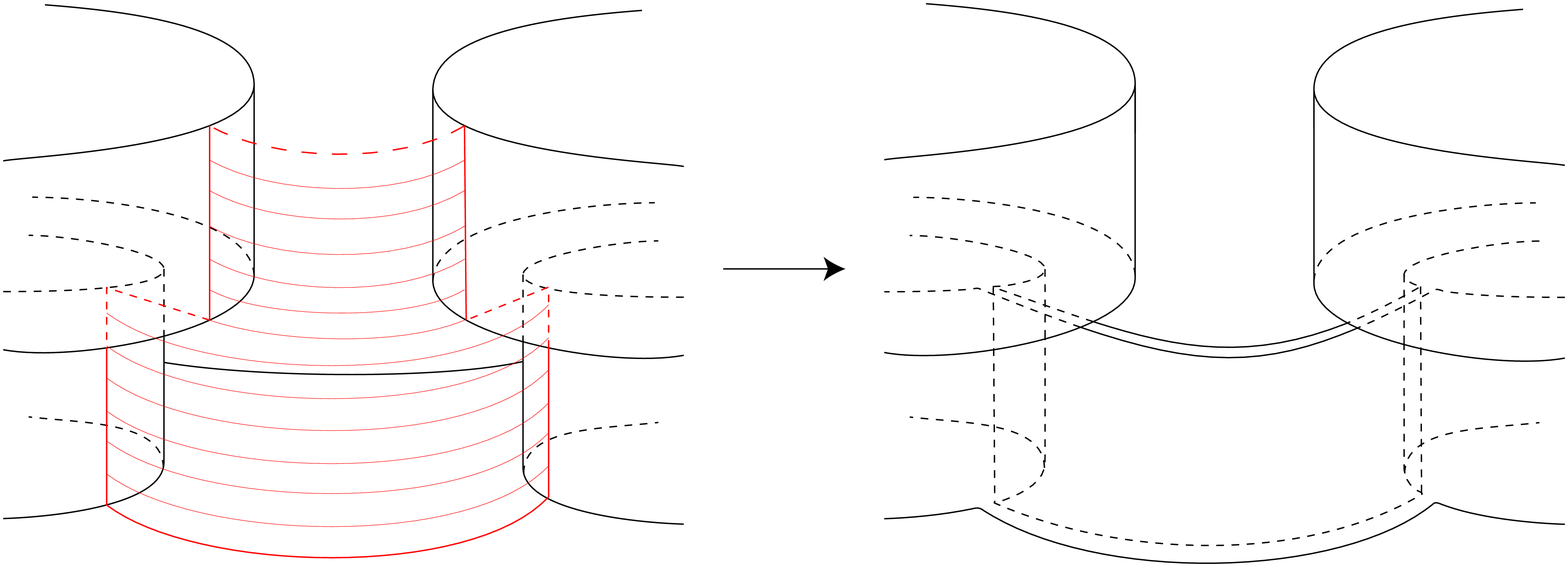}
  \caption{Stopping half way through the band move defined by the red disk defines a surface with a flipped square.}
  \label{flippedbandfig}
  \end{center}
\end{figure}

The horizontal loops and arcs in $S''$ disjoint from the flipped square are precisely those loops and arcs contained in both $S'_0$ and $S'_1$. By assumption, these are all essential, so $S''$ is an essential surface with one flipped square.
\end{proof}

\section{Index-two surfaces}
\label{indxtwosect}

\begin{Lem}
\label{indextwolem}
Every index-two surface is represented by an essential surface in $N$ with up to two flipped disks or one-bridge annuli.
\end{Lem}

\begin{proof}
Because $S$ is an index-two surface, there is a sequence $B_1,\dots,B_\ell$ of K-disks forming an essential loop in the disk complex for $S$. Applying the construction in Lemma~\ref{indexonelem}, we can find for each disk $B_i$ a sequence of surfaces $S^i_j$, each ending in a surface in which $B_i$ is horizontal. 

Consecutive disks $B_i$, $B_{i+1}$ in the sequence are disjoint so the band moves defined by $B_i$, $B_{i+1}$ are disjoint and thus commute with each other.  In other words, performing the first $j$ band moves defined by $B_i$, and the first $k$ band moves defined by $B_{i+1}$ in any order produces the same flat surface, which we will call $S^i_{j,k}$. We will think of these arranged in a rectangle, as in Figure~\ref{squarefig}. Under this convention, $S^{i-1}_{0,j} = S^i_j = S^i_{j,0}$. 

If any one of these surfaces $S^i_{j,k}$ is essential, then we have found our essential representative for $S$. Otherwise, we will choose a horizontal disk $D^i_{j,k}$ for each $S^i_{j,k}$. If $j$ and $k$ are the highest indices in the sequences defined by $B_i$, $B_{i+1}$ then we will let $D^i_{j,h} = B^i$ for $h \leq k$ and $D^i_{h,k} = B^{i+1}$ for $h < j$.
\begin{figure}[htb]
  \begin{center}
  \includegraphics[width=1.5in]{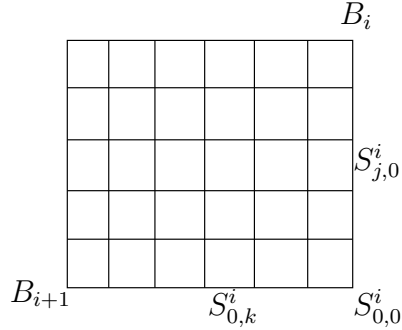}
  \put(0,-10){$S^i_{0,0}$}
  \put(-55,-10){$S^i_{0,k}$}
  \put(-5,100){$B_i$}
  \put(0,45){$S^i_{j,0}$}
  \put(-130,-5){$B_{i+1}$}
  \caption{The labelling for the rectangle of flat surfaces induced by $B_i$, $B_{i+1}$.}
  \label{squarefig}
  \end{center}
\end{figure}

By construction, any two surfaces $S^i_{j,k}$, $S^i_{j+1,k}$ or $S^i_{j,k}$, $S^i_{j,k+1}$ are related by a band move. As in the proof of Lemma~\ref{indexonelem}, one of three things can happen: Either one of the surfaces is essential with a one-bridge annulus, the intermediate surface is essential with a flipped square, or there is a horizontal disk contained in both surfaces. If either of the first two cases occurs, the proof is complete. Otherwise, assume that there is a horizontal disk common to every such pair of surfaces, and choose a disk for each pair. To avoid excessive notation, we will not label this disk, but instead say that it is the disk \textit{associated to} the band move.

Every square of surfaces $S^i_{j,k}$, $S^i_{j+1,k}$, $S^i_{j,k+1}$, $S^i_{j+1,k+1}$ is defined by two disjoint band moves between the four surfaces. If there is no horizontal disk in common with all four surfaces, then, similarly to the index-one case, stopping both band moves mid-way produces an index-two surface with two flipped squares.  If this occurs for any of the squares then, again, the proof is complete.

Thus we will assume that for each such square, there is a horizontal disk $D'^i_{j,k}$ common to all four of the surfaces. The disks defined by the surfaces and by the band moves determine a loop in the disk complex with up to eight vertices and $D'^i_{j,k}$ is disjoint from all these disks (possibly isotopic to one or more of them). If $D'^i_{j,k}$ is not a bridge disk that intersects one of the loop disks in two points in $L$ then $D'^i_{j,k}$ is connected to each vertex of the loop by an edge in the disk complex for $S$. Because the disk complex is flag, this defines a collection of triangles forming disk bounded by the loop. 

If this were the case for each square in the grid defined by the loop $B_1,\dots, B_k$ then the union of disks defined in this way would form an immersed disk bounded by this loop. However, this contradicts the assumption that the loop $B_1,\dots, B_k$ is homotopy non-trivial in the disk complex.

Thus for some square of surfaces, any horizontal disk common to the four flat surfaces is a bridge disk intersecting one of these bridge disks $D''$ in two points in $L$. Let $D''_-$, $D''_+$ be the disks before and after $D''$ in the loop. There are two cases to consider, depending on whether $D''$ is associated to a surface $S^i_{j,k}$ or to a band move.

If $D''$ is associated to a surface, say $S^i_{j,k}$, then the disks $D''_-$, $D''_+$ are associated to band moves and thus are also horizontal disks in $S^i_{j,k}$ (since each is shared by $S^i_{j,k}$ and a second surface.) If these are not bridge disks intersecting in two points then the square between $D'', D''_-, D'^i_{j,k}$ and $D''_+$ can be filled in by two triangles containing an edge from $D''_-$ to $D''_+$. Otherwise, if they are such a pair of bridge disks then $D''_-, D''_+$ and $D'', D'^i_{j,k}$ define two one-bridge annuli, making $S^i_{j,k}$ an index-two surface with two once-punctured annuli.

If $D''$ is shared between two flat surfaces then stopping the band move half way between the two surfaces produces a flat surface with one flipped square, in which the only trivial loops or arcs are the arcs that define $D'^i_{j,k}$ and $D''$. Thus this surface is index-two with one flipped square and one one-bridge annulus.
\end{proof}

After reading this proof, a rough outline for generalizing the theorem to higher indices should be clear. However, for our present purposes, we will only need to understand surfaces of index zero, one and two.

\section{Distance}
\label{distsect}

In this section and the next, we show that the existence of essential flat surfaces in $N$ implies a bound on the displacement distance of an open book decomposition. We will divide the problem into a number of cases, depending on the index of the essential flat surface $S$ and the number of points in $S \cap L$. 

\begin{Lem}
\label{distboundeasylem}
If $S$ is a genus $g$ essential flat surface and $S \cap L \neq \emptyset$ then $d(\pi) \leq \max \{4g, 3\}$.
\end{Lem}

Note that the bound is in terms of the genus of $S$, rather than Euler characteristic.  In other words, the bound does not depend on the number of boundary components of $S$, only on the genus of the surface that results from filling in those boundary components.  In fact, we will see below that having more than two boundary components produces a lower bound than stated.

\begin{proof}
Let $2n$ be the number of points in the intersection $S \cap L$. (The number of points must be even because $S$ is strongly separating, as well as because $L$ is homology trivial in $M$.)

Identify $S^1$ with the interval $[0,1]$ with its endpoints glued together, so that the page $F_0$ does not contain a horizontal subsurface of $S$.  Because $S$ is transverse to the link $L$, we can choose meridian disks for the complement of $N$ (which is a regular neighborhood of $L$) that are disjoint from $S$. This defines a point in each boundary component of each $F_t$ and we will let $m$ be the union of these points.

Let $c_1 < \dots < c_k$ be the levels containing horizontal subsurfaces of $S$. Let $b_1,\dots,b_k$ be levels not containing horizontal subsurfaces such that $c_i < b_i < c_{i+1}$, $b_0 = 0$ and $b_k = 1$.

The identification of $S^1$ with a quotient of $[0,1]$ determines a map $p: F \times [0,1] \rightarrow N$ that sends $F \times \{0\}$ and $F \times \{1\}$ to $F_0$.  We can project the intersection of $S$ with each $F_t$ into $F$ by taking its preimage in $p$.  This projects $S \cap F_0$ to a collection of loops and arcs and projects $S \cap F_1$ onto its image under the monodromy map $\phi$. We will show that the distance between any loop in $F_0$ and any loop in $F_1$ is at most $\max\{4g, 3\}$.

The projection of each horizontal subsurface $S \cap F_{c_i}$ is (up to isotopy disjoint from the marked points $m$) the union of the projections of $S \cap F_{b_i}$ and $S \cap F_{b_{i-1}}$.  Thus these two collections of loops project to pairwise disjoint (up to isotopy) collections of arcs and loops in $F$.  If we pick an arc or loop from each $S \cap F_{c_i}$, the result will thus be a path in $\mC(F)$.  

Recall that we have assumed there are $2n > 0$ points in $S \cap L$. The intersection $S \cap \partial (F_t)$ must contain $2n$ components for each $i$, so each $S \cap F_{b_i}$ contains $n$ arcs and some number of loops. Let $\{\alpha_{0,0}, \alpha_{0,1},\dots,\alpha_{0,n}\}$ be the arcs in $S \cap F_{b_1}$.  For each $i \leq k$, we will label the arcs $\{\alpha_{i,j}\}$ such that if there is an arc in $F_{b_i} \cap S$ isotopic to $\alpha_{i-1,j}$ then we let $\alpha_{i,j} = \alpha_{i-1,j}$.

Each sequence $(\alpha_{i,j})$ as we fix $j$ and vary $i$ defines a path in the arc complex for $F$. The final arc in the sequence is the image under the monodromy of some $\alpha_{0,j'}$ and is thus disjoint from the image of $\alpha_{0,j}$. Thus $d(\pi)$ is at most the number of arcs in the path $(\alpha_{i,j})$. The arcs $\alpha_{i,j}$ (varying both $i$ and $j$) form parallel families corresponding to the vertical bands in $S$. The length of the shortest path will thus be at most the number of vertical bands divided by $n$ (the number of arcs at each stage.)

The vertical bands and annuli define a decomposition of the surface $S$ along arcs and loops. We can calculate the Euler characteristic of $S \setminus L$ from the complementary pieces by noting that each loop contributes zero to the Euler characteristic, while each arc contributes exactly one. (The endpoints of the arc are not in the open surface $S \setminus L$, so we will not count them.)  Since each arc appears in exactly two complementary pieces, each piece adds its Euler characteristic minus half the number of adjacent arcs to the total Euler characteristic of $S \setminus L$.

Because $S$ is tight, every horizontal disk is adjacent to at least four arcs and contributes $1 - \frac{4}{2} = -1$ to the total. An annulus with one or more arcs in the boundary also contributes at most $-1$ to the total. Any other type of piece contributes strictly less than that. In particular, every piece contributes at most negative one fourth the number of adjacent arcs. Since each arc appears in two pieces, the total number of arcs is at most twice the Euler characteristic of $S \setminus L$.

The Euler characteristic of $S$ is $2 - 2g - 2n$ so for $g \geq 1$, the length of the longest path $(\alpha_{i,j})$ at most $-\frac{4 - 4g - 4n}{n} = \frac{4g + 4n - 4}{n} \leq 4g$. For $g=0$, we consider two caes: If $n = 2$ then $S$ is an annulus, there is exactly one arc at each $b_i$ and any two of these arcs are parallel so $d(\phi) = 0 $. Otherwise, if $n \geq 2$ then $\frac{4n-4}{n} < 4$. Since the length is an integer, it is at most $3$.
\end{proof}

\begin{Lem}
\label{distboundfliplem}
If $S$ is a genus $g$ essential surface with one or more flipped squares (but no one-bridge annuli) and $S \cap L \neq \emptyset$ then $d(\pi) \leq \max\{4g, 3\}$.
\end{Lem}

\begin{proof}
First consider the case when there is a single level containing exactly one flipped square in $S$.  Parameterize $S^1$ so that the flipped square is in level $F_0$. Let $\epsilon > 0$ be small enough that there are no other horizontal subsurfaces between $F_{1-\epsilon}$, $F_1 = F_0$ and $F_\epsilon$. We can think of a flipped square as coming from a band move that has been stopped half way through. By either completing the band move, or undoing the beginning of the move, we can find two different flat surfaces $S_-$, $S_+$. If either of these is essential then Lemma~\ref{distboundeasylem} implies the distance bound that we want. Thus we will assume that neither of these surfaces is essential.

The horizontal loops of $S_-$ or $S_+$result pinching two parallel sides of the flipped square together, then isotope the resulting loops to remove the two resulting bigons. If both $S_-$ and $S_+$ are inessential then pinching along each pair of parallel arcs/loops must create a trivial arc or loop. Pinching a loop to itself along a non-trivial arc cannot produce a trivial loop. If pinching two essential loops together creates a trivial loop or arc then the two original loops/arcs must be parallel. Thus the loops/arcs involved in the flipped square consist of two pairs of parallel loops/arcs so that each type of loop/arc intersects the other type in exactly one point. Such loops are distance exactly two in the curve complex.

Similarly, if two arcs are pinched together in both directions, then each arc on one side intersects each arc on the other in a single point. If on either side an arc is pinched to itself then this creates a trivial arc and an essential loop. Since the essential loop is disjoint from the above and below arcs, the above and below arcs again have distance exactly two.

We conclude that every loop or arc in $S \cap F_\epsilon$ is disjoint from every loop or arc in $S \cap F_{-\epsilon}$, except for up to two pairs of loops/arcs that are distance two. On the other hand, note that the Euler characteristic of the surface $S_-$ that results from resoling the flipped square is the same whether or not the arcs/loops in the intermediate subsurface are essential. The surface $S_- \cap (F \times [-\epsilon, \epsilon])$ contains two horizontal subsurfaces with at least two vertical bands/annuli between them. Thus the subsurface $S \cap (F \times [-\epsilon, \epsilon])$ contributes enough to the Euler characteristic of $S$ to account for the arcs that are distance two.

If $S$ contains a number of levels, each containing a single flipped square then we can cut $S^1$ into a number of intervals at these horizontal subsurfaces. For each flipped square, there is extra Euler characteristic to account for the distance two arcs/loops.

If there is a horizontal subsurface with more than one flipped square, then we can again assume that resolving any one of the flipped squares produces an inessential flat surface. By a similar argument to that above, we find that each loop/arc intersects each flipped square at most once or is disjoint from a loop that misses one flipped square and intersects each remaining flipped square at most once. The total number of intersection points between loops/arcs above and below the horizontal subsurface is the number of flipped squares. The distance between two loops/arcs is at most one plus the log (base two) of the intersection number. On the other hand, by resolving the flipped squares, we find extra Euler characteristic to offset the extra distance, so we again find that the displacement distance is at most $\max\{4g,3\}$.
\end{proof}

\section{Mostly horizontal surfaces}
\label{mosltlyhsect}

In the remaining cases, there exist surfaces that do not imply a distance bound.  For example, let $S$ be a surface that consists of a collection of annuli within a regular neighborhood of $\partial N$ and coincides with two pages $F_t$, $F_s$ outside this regular neighborhood. In the case when $F$ has more than one boundary component, choose the vertical annuli to be on both sides of $F_t$. In the case when there is one boundary component, isotope the vertical annulus (in $M$) to a one-bridge annulus. We leave it as an exercise to the reader to show that for $d(\pi)$ sufficiently high, the resulting surface will have topological index one. In other words, this surface is strongly irreducible, but because such a surface exists no matter what the monodromy is, it cannot bound the distance.

In general, we will say that a flat surface $S \subset N$ is \textit{mostly horizontal} if it can be compressed and isotoped in $N$ to a flat (though not necessarily essential) surface $S'$ such that for some regular neighborhood $C \subset N$ of $\partial N$, the complement $S' \setminus C$ is a union of pages $(F_t \cup F_s) \setminus C$. Equivalently, all the vertical bands and vertical annuli of $S'$ are trivial or parallel into $\partial N$. 

\begin{Lem}
\label{distboundclosedesslem}
If $S$ is a genus $g$ essential flat surface, possibly with one or more flipped squares, $S$ is (topologically) index-zero, -one or -two and $S \cap L = \emptyset$ then either $S$ is mostly horizontal or  $d(\pi) \leq 2g-2$.
\end{Lem}

\begin{proof}
As in Lemma~\ref{distboundeasylem}, we identify $S^1$ with the interval $[0,1]$ with its endpoints glued together, so that the level $0$ does not contain a horizontal subsurface of $S$. Because $S$ is disjoint from $L$, there are no vertical bands in $S$, so we will form a path of loops in the curve complex for $F$.

If every level $F_t \cap S$ contains a loop that is not boundary parallel in $F_t$ then we can form a path as in the proof of Lemma~\ref{distboundeasylem}, and the length of the path will be bounded by the number of horizontal subsurfaces in $S$. No horizontal subsurface is an annulus, so every horizontal subsurface has Euler characteristic at most $-1$, and the length of the path is at most $2g - 2$, the negative Euler characteristic of the genus $g$ surface $S$.

Otherwise, without loss of generality assume $F_0$ is a level such that every loop $F_t \cap S$ is boundary parallel. (This includes the possibility that the intersection is empty.) Outside a regular neighborhood $C \subset N$ of $\partial N$, the page $F_0 \setminus C$ is entirely on either the positive side or the negative side of $S$. If $F_0 \setminus C$ is on the positive side of $S$ then we will say that $F_0$ is \textit{mostly above} $S$. If it is on the negative side of $S$, we will say it is \textit{mostly below} $S$.

Without loss of generality, assume $F_0$ is mostly above $S$. If there is a second level $F_t$ that is mostly below $S$ then the intersection of $S$ with $F \times [0,t] \setminus C$ separates the top and bottom of the surface product. The only incompressible surfaces in $F \times [0,1]$ are copies of pages or have boundary that intersects $F \times \{0,1\}$ in essential loops and arcs. If we maximally compress $S$ within $F \times [0,t] \setminus C$, the resulting surface will be incompressible and will separate $F \times \{0\}$ from $F \times \{0\}$, so it must be isotopic to $F_s$ for some $s \in (0,t)$. The same argument applies to the intersection of $S$ with $F \times [t, 1] \setminus C$. Thus if one page of the open book is mostly above $S$ and another is mostly below then $S$ can be compressed down to two copes of $F$, i.e.\ $S$ is a mostly horizontal surface.

If there is no page $F_t$ mostly below $S$ then $S$ must be compressible, since it is contained in $F \times [0,1]$, its boundary is isotopic into $\partial F \times [0,1]$ but the surface does not consist of a union of pages. In other words, $S$ does not have index zero, so it has index one or two  by assumption. The set of disks on one side of a compressible form a contractible disk complex~\cite[Theorem 5.3]{mccullough} so any surface with compressing disks on only one side does not have a well defined index. Since $S$ has index zero, one or two by assumption, $S$ must be compressible to both sides.

Let $\mathcal{D}$ be a maximal set of compressing disks on the positive side of $S$. Let $S_0,\dots,S_k$ be a sequence of surfaces resulting from band moves defined by the disks $\mathcal{D}$, after compressing along any trivial vertical annulus that results from a band move. (Thus each $S_i$ will be an essential flat surface.) The final surface $S_k$ will be either incompressible or empty. After each band move, one of two things may happen: There may no longer be any pages mostly above $S$ or there may be a new page that is mostly below $S$.

If, after the band move producing $S_i$, the last page mostly above $S_{i-1}$ is removed before any pages mostly below $S_i$ appear, then $S_i$ is an essential flat surface with non-boundary-parallel arcs/loops at every level and we get the distance bound of $2g-2$. If, on the other hand, a page becomes mostly below $S_i$ before the last page mostly above $S_{i-1}$ is eliminated then $S_i$ will be mostly horizontal, by the second argument above. Since $S_i$ is either isotopic to $S$ or the result of compressing $S$, this implies $S$ is also mostly horizontal.

If, however, the last page mostly above $S_{i-1}$ is eliminated in the same step that some page becomes mostly below $S_i$, we can stop the band move half way through, to find a flat surface $S'$ with an additional flipped square, such that no page is mostly above or below $S'$. Having one or more flipped squares does not affect the distance bound so in this case, following the proof of Lemma~\ref{distboundfliplem}, we still get the distance bound $2g-2$.
\end{proof}

We can now combine the different arguments described so far to prove the most general version of this Lemma:

\begin{Lem}
\label{distboundfinallem}
If $S$ is a genus $g$ essential surface, possibly with one or more flipped squares and up to two one-bridge annuli and $S$ is (topologically) index-zero, -one or -two then either $S$ is mostly horizontal or $d(\pi) \leq 4g+8$. For zero or one flipped squares or one-bridge annuli, the inequality becomes $d(\pi) \leq 4g+4$.
\end{Lem}

\begin{proof}
Because this proof combines arguments from the previous Lemmas, we will give only a rough description of the proof. Let $2n$ be the number of points in $S \cap L$ and let $k$ be the number of one-bridge annuli. If $n > k$ then there are at least $n - k$ essential horizontal arcs between consecutive horizontal subsurfaces, so we can apply the argument from Lemma~\ref{distboundeasylem}. The Euler characteristic of $S$ is $2 - 2g - 2n$ and there are $n - k$ paths of arcs, so the length of the longest path $(\alpha_{i,j})$ is at most $-\frac{4 - 4g - 4n}{n-k}$. By assumption, $k$ is at most $2$. When $g \geq 1$, the total is at most  $4g + 8$, which happens when $k = 2$ and $n = 3$. For $g = 0$, the total is at most $8$. If we do the same calculations with $n \leq k+1$, we find that the bound is $4g+4$ for $g \geq 1$ and $4$ for $g = 0$.

Otherwise, we must have $n = k$. In this case, there may be a level $F_t$ such that $S \cap F_t$ does not contain an essential loop or arc. We apply the argument in the proof of Lemma~\ref{distboundclosedesslem} to find either an essential subsurface with essential loops/arc at every level or an essential surface $S'$ in which there is one page mostly above $S'$ and a second page mostly below $S'$. In the first case, we get the distance bound of $2g + 6$. In the second case, we find that $S$ is mostly horizontal. Thus for any values of $n$, $k$, the distance is at most $4g+8$.
\end{proof}

\section{The main theorems}
\label{proofsect}

\begin{proof}[Proof of Theorem~\ref{mainthm2}]
Let $d$ be the displacement of the monodromy of an open book decomposition for $M$. Let $(\Sigma, H^-_\Sigma, H^+_\Sigma)$ be a genus $g$ Heegaard splitting for $M$ such that $d > 4g+4$. Let $E$ be a thin path in $\mS(M,L)$ representing $(\Sigma, H^-_\Sigma, H^+_\Sigma)$. 

The path $E$ has an interior maximum $v$ that represents a surface $S$ isotopic to $\Sigma$. This maximum has index one in $\mS(M, L)$ so by Lemma~\ref{indexonelem}, $v$ is represented by an essential surface with at most one flipped disk or a one-bridge annulus. By Lemma~\ref{distboundfinallem}, this implies that either $d \leq 4g + 4$ or $S$ can be compressed to a surface isotopic to the union of two pages of the surface bundle. By assumption, $d > 4g+4$ so the latter condition must be satisfied and $\Sigma$ can be compressed down to the Heegaard surface induced by $(L, \pi)$. By Proposition~22 in~\cite{me:stabs}, this implies that $\Sigma$ is in fact a stabilization of the Heegaard surface induced by the open book decomposition.

We next need to show that every automorphism of $\Sigma$ is induced by this open book. Every element of $Isot(M,\Sigma)$ is defined by an ambient isotopy of $M$ (not necessarily fixing $L$) that takes $\Sigma$ off iteself, then back to itself. This isotopy defines a sequence of paths in $\mS(M, \Sigma)$. By Lemma~\ref{indextwopathlem}, this implies that there is a path in the complex of surfaces with index-two maxima and index-one and -zero minima, all with genus $g$ and related by bridge compressions. Each can be isotoped to an essential flat surface. Because of the distance bound, all of these surfaces must be mostly horizontal.

Note that the Euler characteristic of $\Sigma$ is exactly twice that of $F$, so outside a regular neighborhood $C$ of $L$, each surface consists of two pages of the open book. Within each solid torus component of $C$, there is either a pair of bridge disks defining a one-bridge annulus for $\Sigma$ or an essential annulus with one boundary loop in $L$ and the other in $\Sigma$. Since $\Sigma \setminus C$ is essential, these bridge disks/annuli are unique and the isotopy can be extended to them.

Moreover, the isotopy can be further extended to a regular neighborhood of the disks and annuli, which is isotopic to $C$. Thus the isotopy of $\Sigma$ is defined by an isotopy of two pages of the surface bundle $N \setminus C$. The only such isotopy is the one that spins the pages around the surface bundle structure so the isotopy of $\Sigma$ is the book rotation defined by the open book decomposition.
\end{proof}

\begin{proof}[Proof of Theorem~\ref{mainthm}]
This theorem follows immediately from Theorem~\ref{mainthm2}: Every open book decomposition for $M$ induces a Heegaard splitting for $M$ whose genus is $g = 1 - \chi$, where $\chi$ is the Euler characteristic of a page. Since every Heegaard splitting not induced by the open book has genus $h$ satisfying $4h+4 \geq d > 4g+8$, the Euler characteristic of every open book for $M$ is at least that of $(L, \pi)$. If an open book decomposition induces the same Heegaard splitting as $(L, \pi)$ then by Theorem~\ref{mainthm2}, it induces an automorphism of the Heegaard splitting whose fixed set is the binding of this open book. However, Theorem~\ref{mainthm2} states that every automorphism of the Heegaard splitting is induced by $(L, \pi)$ and thus has fixed set $L$ with pages contained in the Heegaard surface. Because the bindings are the same and they have two pages in common, the two open book decompositions for $M$ are isotopic to each other.
\end{proof}

\bibliographystyle{amsplain}
\bibliography{books}

\providecommand{\bysame}{\leavevmode\hbox to3em{\hrulefill}\thinspace}
\providecommand{\MR}{\relax\ifhmode\unskip\space\fi MR }
\providecommand{\MRhref}[2]{%
  \href{http://www.ams.org/mathscinet-getitem?mr=#1}{#2}
}
\providecommand{\href}[2]{#2}
\begin{thebibliography}{10}

\bibitem{akoz}
Selman Akbulut and Burak Ozbagci, \emph{Lefschetz fibrations on compact {S}tein
  surfaces}, Geom. Topol. \textbf{5} (2001), 319--334 (electronic). \MR{1825664
  (2003a:57055)}

\bibitem{bach:index}
David Bachman, \emph{Topological index theory for surfaces in 3-manifolds},
  Geom. Topol. \textbf{14} (2010), no.~1, 585--609. \MR{2602846 (2011f:57042)}

\bibitem{bs:versus}
David Bachman and Saul Schleimer, \emph{Surface bundles versus {H}eegaard
  splittings}, Comm. Anal. Geom. \textbf{13} (2005), no.~5, 903--928.
  \MR{2216145 (2006m:57027)}

\bibitem{gabai}
David Gabai, \emph{Foliations and the topology of {$3$}-manifolds. {III}}, J.
  Differential Geom. \textbf{26} (1987), no.~3, 479--536. \MR{910018
  (89a:57014b)}

\bibitem{giroux}
Emmanuel Giroux, \emph{G\'eom\'etrie de contact: de la dimension trois vers les
  dimensions sup\'erieures}, Proceedings of the {I}nternational {C}ongress of
  {M}athematicians, {V}ol. {II} ({B}eijing, 2002) (Beijing), Higher Ed. Press,
  2002, pp.~405--414. \MR{1957051 (2004c:53144)}

\bibitem{hs:thin}
Chuichiro Hayashi and Koya Shimokawa, \emph{Thin position of a pair
  (3-manifold, 1-submanifold)}, Pacific J. Math. \textbf{197} (2001), no.~2,
  301--324. \MR{1815259 (2002b:57020)}

\bibitem{me:stabs}
Jesse Johnson, \emph{Bounding the stable genera of {H}eegaard splittings from
  below}, J. Topol. \textbf{3} (2010), no.~3, 668--690. \MR{2684516
  (2011g:57024)}

\bibitem{axiomatic}
\bysame, \emph{Computing isotopy classes of {H}eegaard splittings}, preprint
  (2010), arXiv:1004.4669.

\bibitem{me:stabs2}
\bysame, \emph{{An upper bound on common stabilizations of Heegaard
  splittings}}, preprint (2011), arXiv:1107.2127.

\bibitem{jmm}
Jesse Johnson, Yair Minsky, and Yoav Moriah, \emph{Heegaard splittings with
  large subsurface distances}, preprint (2010).

\bibitem{loipi}
Andrea Loi and Riccardo Piergallini, \emph{Compact {S}tein surfaces with
  boundary as branched covers of {$B^4$}}, Invent. Math. \textbf{143} (2001),
  no.~2, 325--348. \MR{1835390 (2002c:53139)}

\bibitem{ms:holes}
Howard Masur and Saul Schleimer, \emph{The geometry of the disk complex},
  preprint (2010).

\bibitem{mccullough}
Darryl McCullough, \emph{Virtually geometrically finite mapping class groups of
  {$3$}-manifolds}, J. Differential Geom. \textbf{33} (1991), no.~1, 1--65.
  \MR{1085134 (92c:57001)}

\bibitem{st:dist}
Martin Scharlemann and Maggy Tomova, \emph{Alternate {H}eegaard genus bounds
  distance}, Geom. Topol. \textbf{10} (2006), 593--617 (electronic).
  \MR{2224466 (2007b:57040)}

\bibitem{Schl}
Jennifer Schultens, \emph{Heegaard splittings of {S}eifert fibered spaces with
  boundary}, Trans. Amer. Math. Soc. \textbf{347} (1995), no.~7, 2533--2552.
  \MR{1297537 (95j:57019)}

\end{thebibliography}

\end{document}